\documentclass[11pt]{amsart}
\usepackage{amsmath,mathrsfs, amssymb, amsthm}
\usepackage[all]{xy}
\usepackage{hyperref} 
\hypersetup{ 
colorlinks=true, 
linkbordercolor={1 1 1}, 
citebordercolor={1 1 1} 
} 

\newcommand{\linenumberscmd}{}

\newtheorem{thm}{Theorem}[section]
\newtheorem{cor}[thm]{Corollary}
\newtheorem{lem}[thm]{Lemma}
\newtheorem{prop}[thm]{Proposition}
\newtheorem{proposition}[thm]{Proposition}

\theoremstyle{definition}
\newtheorem{defn}[thm]{Definition}
\newtheorem{quest}[thm]{Question}

\newtheorem{rem}[thm]{Remark} 
\numberwithin{equation}{thm}

\newcommand{\N}{\mathbb{N}}
\newcommand{\setN}{\mathbb{N}}

\newcommand{\res}{\mathbin{\upharpoonright}}

\newcommand{\convergesto}{\mathbin{\downarrow =}}

\newcommand{\M}{\mathscr{M}}
\newcommand{\Lang}{\mathsf{L}}

\newcommand{\emp}{\emptyset}

\newcommand{\RCA}{\mathsf{RCA}_0}
\newcommand{\RCAo}{\mathsf{RCA}_0}
\newcommand{\ACA}{\mathsf{ACA}_0}
\newcommand{\ACAo}{\mathsf{ACA}_0}

\newcommand{\WKLo}{\mathsf{WKL}_0}

\newcommand{\CA}{\mathsf{CA}_0}

\newcommand{\QF}{\mathsf{QF}}

\newcommand{\FCP}{\mathsf{FCP}}

\newcommand{\NCE}{\mathsf{NCE}}
\newcommand{\cl}{\operatorname{cl}}

\newcommand{\AL}{\mathsf{CE}}

\let\oldsetminus-

\newcommand{\chat}[1]{\widehat{#1}}

\title{Reverse mathematics and properties of finite character}
\author{Damir D. Dzhafarov}
\address{\hspace*{-\parindent}Department of Mathematics\\
  University of Notre Dame\\
Department of Mathematics\\
University of Notre Dame\\
255 Hurley Hall\\
Notre Dame, Indiana 46556 USA} \email{ddzhafar@nd.edu}

\author{Carl Mummert}
\address{\hspace*{-\parindent}Department of Mathematics\\
  Marshall University\\
  1 John Marshall Drive\\
  Huntington, West Virginia 25755 USA} \email{mummertc@marshall.edu}
\thanks{The authors are grateful to Denis Hirschfeldt, Antonio
  Montalb\'{a}n, and Robert Soare for valuable comments and
  suggestions, and to an anonymous referee who suggested an 
improvement that strengthened Proposition~\ref{P:NCE_ideals}. 
 The first author was partially supported by an NSF
  Graduate Research Fellowship and an NSF Postdoctoral Fellowship.}

\subjclass[2010]{Primary 03B30, 03F35, Secondary 03E25}
\keywords{Reverse mathematics, finite character, axiom of choice, closure operator}

\date{September 15, 2011} 

\begin{document}

\begin{abstract}
  We study the reverse mathematics of the principle stating that, for every
  property of finite character, every set has a
  maximal subset satisfying the property. In the context of set
  theory, this variant of Tukey's lemma is equivalent to the axiom of
  choice. We study its behavior in the context of second-order 
arithmetic, where it applies to sets of natural numbers only, 
and give a full characterization of
its strength in terms of the quantifier structure of the formula
  defining the property. We then study the interaction between
  properties of finite character and finitary closure operators, and
  the interaction between these properties and a class of
  nondeterministic closure operators.
\end{abstract}

\maketitle

\tableofcontents

\linenumberscmd

\section{Introduction} 

A formula $\varphi$ with one free set variable is
of {\em finite character}, and has the {\em finite
character property}, if $\varphi(\emp)$ holds and, for every set $A$,
$\varphi(A)$ holds if and only if $\varphi(F)$ holds for every finite
$F \subseteq A$. In this paper, we restrict our attention to formulas of second-order
arithmetic, and consider several variants and restrictions of
the principle $\FCP$ (Definition~\ref{def:fcp}) which asserts that for
every formula of finite character, every subset of $\setN$ has a
maximal subset satisfying that formula.  Because the empty set
satisfies any formula of finite character, the soundness 
of this principle in second-order arithmetic can be verified in $\mathsf{ZFC}$ by straightforward 
application of Zorn's lemma. Detailed definitions of second-order arithmetic and
the subsystems studed in this paper 
are given by Simpson~\cite{Simpson-2009}.

The principle $\AL$ 
(Definition~\ref{def:ce})  asserts that given sets $A \subseteq B
\subseteq \setN$, a formula $\varphi$ of finite character and a finitary
closure operator $D$, such that $A$ is a $D$-closed set satisfying the
formula, there is a set $X$ which is maximal with respect to the
conditions that $A \subseteq X \subseteq B$, $\varphi(X)$ holds, and $X$ is
$D$-closed.  In the third section, we give a full characterization of the strength of fragments of $\AL$ in terms of the complexity of the formulas of finite character to which they apply.

We can further generalize $\AL$ by
replacing the finitary closure operator with a more general kind of operator
which we name a \textit{nondeterministic closure operator}. The corresponding 
principle, $\NCE$ (Definition~\ref{D:NCE}), is studied in the final section, where a full characterization of its strength is obtained. 

\newcommand{\FIP}{\mathsf{FIP}}

We were led to study the reverse mathematics of $\FCP$ by our separate
work~\cite{DM-201X} on the principle $\FIP$ which states that every
countable family of subsets of $\mathbb{N}$ has a maximal subfamily
with the finite intersection property. All the
principles studied there are consequences of appropriate restrictions
of $\FCP$. Similarly,
Propositions~\ref{p:alextend}~and~\ref{P:NCE_ideals} below demonstrate
how $\AL$ and $\NCE$ can be used to prove facts about countable
algebraic objects in second-order arithmetic.  In light of these
applications, we find it worthwile to have a complete understanding of
the reverse mathematics strengths of these principles.

Considering this paper together 
with our work on $\FIP$ gives a new example of two principles, $\FCP$ and
$\FIP$, which are each equivalent to the axiom of choice when formalized in set theory, but which have drastically different strengths when formalized in second-order arithmetic. The axiom scheme for $\FCP$ is equivalent to full comprehension in second-order arithmetic, while $\FIP$ is weaker than $\ACAo$ and incomparable with $\WKLo$. 

\section{Properties of finite character}\label{s:pfc}

We begin with the study of various forms of the following principle.

\begin{defn} \label{def:fcp}
The following scheme is defined in $\RCA$.
\begin{list}{\labelitemi}{\leftmargin=0em}\itemsep2pt
\item[]($\FCP$) For each $\Lang_2$ formula $\varphi$ of finite character, which
may have arbitrary set parameters, every set $A$ has a
\mbox{$\subseteq$-maximal} subset $B$ such that $\varphi(B)$ holds.
\end{list}
\end{defn}

\noindent  $\FCP$ is analogous to the set-theoretic principle 
$\mathsf{M} \, 7$ in
the catalog of Rubin and Rubin~\cite{RR-1985}, which is equivalent to
the axiom of choice~\cite[p.~34 and Theorem~4.3]{RR-1985}.

In order to better gauge the reverse mathematical strength of $\FCP$,
we consider restrictions of the formulas to which it applies.  As with
other such ramifications, we will primarily be interested in
restrictions to classes in the arithmetical and analytical
hierarchies.  In particular, for each $i \in \{0,1\}$ and $n \geq 0$,
we make the following definitions:
\begin{itemize}
\item $\Sigma^i_n\text{-}\FCP$ is the restriction of $\FCP$ to
$\Sigma^i_n$ formulas;

\item $\Pi^i_n\text{-}\FCP$ is the restriction of $\FCP$ to $\Pi^i_n$
formulas;

\item $\Delta^i_n\text{-}\FCP$ is the scheme which says that for every
$\Sigma^i_n$ formula $\varphi(X)$ and every $\Pi^i_n$ formula
$\psi(X)$, if $\varphi(X)$ is of finite character and
\[ 
(\forall X)[\varphi(X) \Longleftrightarrow \psi(X)],
\]
then every set $A$ has a $\subseteq$-maximal set $B$ such that
$\varphi(B)$ holds.

\end{itemize} 
We also define $\QF\text{-}\FCP$ to be the restriction of $\FCP$ to the
class of quantifer-free formulas without parameters.

The following proposition demonstrates two monotonicity properties of formulas 
of finite character.

\begin{prop}\label{p:fcmonotone}
Let $\varphi(X)$ be a formula of finite character. The following are provable in $\RCAo$:
\begin{enumerate}
\item if $A \subseteq B$ and $\varphi(B)$ holds then $\varphi(A)$
holds;
\item if $A_0 \subseteq A_1 \subseteq A_2 \subseteq \cdots$ is a
sequence of sets such that $\varphi(A_i)$ holds for each $i\in\setN$,
and $\bigcup_{i \in \setN} A_i$ exists, then
 $\varphi(\bigcup_{i \in \setN} A_i)$ holds.
\end{enumerate}
\end{prop}
\begin{proof}
  The proof of (1) is immediate from the definitions. For (2), the key
  point is to show that if $F$ is a finite subset of $\bigcup_{i \in
    \setN} A_i$ then there is some $j\in\setN$ with $F \subseteq
  A_j$. This follows from induction on the $\Sigma^0_1$ formula
  $\psi(n,F) \equiv (\exists m)(\forall i < n)(i \in F \Longrightarrow i
  \in A_m)$, in which $F$ is a set parameter.
\end{proof}

Our first theorem in this section characterizes 
most of the above restrictions of
$\FCP$ (see Corollary~\ref{c:fcpstrength}).  We draw particular attention to
part~(2) of the theorem, where $\Sigma^0_1$ does not appear in the list of
classes of formulas.  The reason behind this will be made apparent by
Theorem~\ref{P:Sig1_RCA}.

\begin{thm}\label{thm_main_fcp} 
For $i \in \{0,1\}$ and $n \geq 1 $,
let $\Gamma$ be any of $\Pi^i_n$, $\Sigma^i_n$, or~$\Delta^i_n$.
\begin{enumerate}
\item $\Gamma$-$\FCP$ is provable in $\Gamma$-$\CA$;

\item If\/ $\Gamma$ is $\Pi^0_n$, $\Pi^1_n$, $\Sigma^1_n$, or
$\Delta^1_n$, then $\Gamma$-$\FCP$ implies $\Gamma$-$\CA$
over~$\RCAo$.

\end{enumerate}
\end{thm}

The proof of this theorem will make use of the following technical lemma, which is needed only because there are no term-forming operations for sets in the
language $\Lang_2$ of second-order arithmetic. For example, there is no term in $\Lang_2$ that takes a set
$X$ and a number $n$ and returns $X \cup D_n$ where, as in the 
rest of this paper, $D_n$ denotes the finite set with canonical index $n$, or $\emptyset$ if $n$ is not a canonical index. 
The moral of
the lemma is that such terms can be interpreted into $\Lang_2$ in a
natural way.

The coding of finite sets by their canonical indices can be formalized
in $\RCAo$ in such a way that the predicate $i \in D_n$ is defined by
a formula $\rho(i,n)$ with only bounded quantifiers, and such that the
set of canonical indices is also definable by a bounded-quantifier
formula~\cite[Theorem II.2.5]{Simpson-2009}. Moreover, $\RCAo$ proves
that every finite set has a canonical index. We use the notation $Y =
D_n$ to abbreviate the formula $(\forall i)[ i \in Y \Longleftrightarrow
\rho(i,n)]$, along with similar notation for subsets of finite sets.

\begin{lem}\label{l:finiteset} 
Let $\varphi(X)$ be a formula with one
free set variable. There is a formula $\chat{\varphi}(x)$ with one
free number variable such that $\RCA$ proves
\begin{equation}\label{e:finiteset} 
(\forall A)(\forall n)[ A = D_n
  \Longrightarrow (\varphi(A) \Longleftrightarrow \chat{\varphi}(n))].
\end{equation}
Moreover, we may take $\chat{\varphi}$ to have the same
complexities in the arithmetical and analytic hierarchies
as~$\varphi$.
\end{lem}

\begin{proof} 
Let $\rho(i,n)$ be the formula defining the relation $i
\in D_n$, as discussed above. We may assume $\varphi$ is written in
prenex normal form. Form $\chat{\varphi}(n)$ by replacing each
occurrence $t \in X$ of $\varphi$, $t$ a term, with the formula
$\rho(t,n)$.

Let $\psi(X, \bar{Y}, \bar{m})$ be the quantifier-free matrix of
$\varphi$, where $\bar{Y}$ and $\bar{m}$ are sequences of variables
that are quantified in~$\varphi$. Similarly, let
$\chat{\psi}(n,\bar{Y},\bar{m})$ be the matrix of $\chat{\varphi}$.
Fix any model $\M$ of $\RCA$ and fix $n,A \in \M$ such that $\M
\models A = D_n$. A straightforward metainduction on the structure of
$\psi$ proves that
\[
\M \models (\forall \bar{Y})(\forall \bar{m})[ \psi(A, \bar{Y},
\bar{m}) \Longleftrightarrow \chat{\psi}(n, \bar{Y}, \bar{m})].
\] 

The key point is that the atomic formulas in $\psi(A, \bar{Y},
\bar{m})$ are the same as those in $\chat{\psi}(n,\bar{Y},\bar{m})$,
with the exception of formulas of the form $t \in A$, which have been
replaced with the equivalent formulas of the form~$\rho(t,n)$.

A second metainduction on the quantifier structure of $\varphi$ shows
that we may adjoin quantifiers to $\psi$ and $\chat{\psi}$ until we
have obtained $\varphi$ and $\chat{\varphi}$, while maintaining
logical equivalence.  Thus every model of $\RCA$
satisfies~(\ref{e:finiteset}).

Because $\rho$ has only bounded quantifiers, the substitution required
to pass from $\varphi$ to $\chat{\varphi}$ does not change the
complexity of the formula.
\end{proof}

We shall sometimes identify a finite set with its canonical index. Thus, if~$F$ is finite and~$n$ is its canonical index, we may
write~$\chat{\varphi}(F)$ for~$\chat{\varphi}(n)$.

\begin{proof}[Proof of Theorem~\ref{thm_main_fcp}] 
For~(1), let $\varphi(X)$ and $A = \{a_i : i \in \N\}$ be an
instance of $\Gamma$-$\FCP$. Define $g \colon 2^{<\N} \times \N \to
\{0,1\}$ by
\[ 
g(\tau,i) =
\begin{cases} 
  1 & \text{if } \chat{\varphi}(\{ a_j : \tau(j) \convergesto 1\} 
       \cup \{a_i\}) \text{ holds},\\ 
0 & \text{otherwise}.
\end{cases}
\] 
where $\chat{\varphi}$ is as in the lemma.
 The function $g$ exists by $\Gamma$
comprehension.  By primitive recursion, there exists a function $h
\colon \N \to \{0,1\}$ such that for all $i \in \N$, $h(i) = 1$ if and
only if $g(h \res i, i) = 1$.  For each $i \in \N$, let $B_i = \{a_j :
j < i \land h(j) = 1\}$.  An induction on $\varphi$ shows that
$\varphi(B_i)$ holds for every $i\in \N$.

Let $B = \{ a_i : h(i) = 1\} = \bigcup_{i \in \N} B_i$.  Because
Proposition~\ref{p:fcmonotone} is provable in $\RCA$ and hence in
$\Gamma\text{-}\CA$, it follows that $\varphi(B)$ holds.  By the same
token, if $\varphi(B \cup \{a_k\})$ holds for some $k$ then so must
$\varphi(B_k \cup \{a_k\})$, and therefore $a_k \in B_{k+1}$, which
means that $a_k \in B$.  Therefore $B$ is $\subseteq$-maximal, and we
have shown that $\Gamma$-$\CA$ proves $\Gamma$-$\FCP$.

For~(2), we assume $\Gamma$ is one of $\Pi^0_n$, $\Pi^1_n$, or
$\Sigma^1_n$; the proof for $\Delta^1_n$ is similar. We work in $\RCA
+ \Gamma\text{-}\FCP$.  Let $\varphi(n)$ be a formula in $\Gamma$ and
let $\psi(X)$ be the formula $(\forall n)[n \in X \Longrightarrow \varphi(n)].$ It
is easily seen that $\psi$ is of finite character, and it belongs to
$\Gamma$ because $\Gamma$ is closed under universal number quantification. 
By $\Gamma$-$\FCP$, $\N$ contains a $\subseteq$-maximal
subset $B$ such that $\psi(B)$ holds. For any~$y$, if $y \in B$ then
$\varphi(y)$ holds.  On the other hand, if $\varphi(y)$ holds then so
does $\psi(B \cup \{y\})$, so $y$ must belong to $B$ by maximality.
Therefore $B = \{y \in \N : \varphi(y) \}$, and we have shown that
$\Gamma$-$\FCP$ implies~$\Gamma$-$\CA$.
\end{proof}

The corollary below summarizes the theorem as it applies to the
various classes of formulas we are interested in.  Of special note
is part~(5), which says that $\FCP$ itself (that is, $\FCP$ for
arbitrary $\Lang_2$-formulas) is as strong as any theorem of
second-order arithmetic can be.

\begin{cor}\label{c:fcpstrength}
The following are provable in $\RCA$:
\begin{enumerate}
\item $\Delta^0_1$-$\FCP$, $\Sigma^0_0\text{-}\FCP$, and
$\QF\text{-}\FCP$;

\item for each $n \geq 1$, $\ACA$ is equivalent to $\Pi^0_n$-$\FCP$;

\item for each $n \geq 1$, $\Delta^1_n\text{-}\CA$ is equivalent to
$\Delta^1_n$-$\FCP$;

\item for each $n \geq 1$, $\Pi^1_n$-$\CA$ is equivalent to\/
$\Pi^1_n$-$\FCP$ and to $\Sigma^1_n$-$\FCP$;

\item $\mathsf{Z}_2$ is equivalent to $\FCP$.

\end{enumerate}
\end{cor}

The case of $\FCP$ for $\Sigma^0_1$ formulas is anomalous.  The proof
of part (2) of Theorem~\ref{thm_main_fcp} does not go through for $\Sigma^0_1$
because this class is not closed under universal quantification.  As
 the next theorem shows, this limitation is quite
significant.  Intuitively, the proof uses the fact that a $\Sigma^0_1$ formula $\varphi$ is continuous in the sense that if 
$\varphi(X)$ holds then there is an $N$ such that $\varphi(Y)$ holds for any
$Y$ with $X \cap \{0, \ldots, N\} = Y \cap \{0, \ldots, N\}$. 

\begin{thm}\label{P:Sig1_RCA} 
$\Sigma^0_1$-$\FCP$ is provable in $\RCA$.
\end{thm}

\begin{proof}
Let $\varphi(X)$ be a $\Sigma^0_1$ formula of finite
character.  We claim that there exists some $c_\varphi \in \setN$ such
 that for every set $A$, if $A \cap \{0, \ldots, c_\varphi\} = \emptyset$ then
$\varphi(A)$ holds.  To show this, put $\varphi(X)$ in normal form, so that
\[
\varphi(X) \equiv (\exists m)\rho(X[m])
\]
where $\rho$ is $\Sigma^0_0$. As $\varphi(\emptyset)$ holds, there is some 
$c = c_\varphi$  
such that $\rho(\emptyset[c])$ holds. Now let $A$ be any set such that
$A \cap \{0, \ldots, c\} = \emptyset$. Then $\rho(A[c])$ holds, 
so $\varphi(A)$ holds. This proves the claim.

Now fix any set~$A$.  By the claim, we know that $\varphi(A - \{0,
\ldots, c_\varphi\})$ holds. We may use bounded $\Sigma^0_1$
comprehension~\cite[Theorem~II.3.9]{Simpson-2009} to form the set $I$
of $m$ such that $D_m \subseteq \{0, \ldots, c_\varphi\}$ and
$\varphi(D_m \cup (A - \{0, \ldots, c_\varphi\}))$ holds.  We may then
choose $m \in I$ such that $D_m$ has maximal cardinality among the
sets with indices in~$I$.  It follows immediately that $D_m \cup (A -
\{0, \ldots, c_\varphi\})$ is a maximal subset of $A$ satisfying~$\varphi$.
\end{proof}

The above proof contains an implicit non-uniformity in choosing a finite 
set of maximal cardinality. The next proposition shows that this
non-uniformity is essential, by showing that a sequential form of
$\Sigma^0_1\text{-}\FCP$ is a strictly stronger principle.

\newpage

\begin{prop}\label{P:fcp-uniform}
The following are equivalent over $\RCA$:
\begin{enumerate}
\item $\ACA$;
\item for every family $A = \langle A_i : i \in \N \rangle$ of sets,
and every $\Sigma^0_1$ formula $\varphi(X,x)$ with one free set
variable and one free number variable such that for all $i \in \N$,
the formula $\varphi(X,i)$ is of finite character, there exists a
family $B = \langle B_i : i \in \N \rangle$ of sets such that for all
$i$, $B_i$ is a $\subseteq$-maximal subset of $A_i$ satisfying
$\varphi(X,i)$.
\end{enumerate}
\end{prop}

\begin{proof} The forward implication follows by a straightforward
modification of the proof of Theorem~\ref{thm_main_fcp}.  For the
reversal, let a one-to-one function $f \colon \N \to \N$ be given.
For each $i \in \N$, let $A_i = \{i\}$, and let $\varphi(X,x)$ be the
formula
\[ 
(\exists y)[x \in X \Longrightarrow f(y) = x].
\] 
Then, for each $i$, $\varphi(X,i)$ has the finite character
property, and for every set $S$ that contains $i$, $\varphi(S,i)$
holds if and only if $i \in \operatorname{range}(f)$.  Thus, if $B =
\langle B_i : i \in \N \rangle$ is the subfamily obtained by applying
part (2) to the family $A = \langle A_i : i \in \N \rangle$ and the
formula $\varphi(X,x)$, then
\[ 
i \in \operatorname{range}(f) \Longleftrightarrow B_i = \{i\}
\Longleftrightarrow i \in B_i.
\] 
It follows that the range of $f$ exists.
\end{proof}

\begin{rem}
  Proposition~\ref{P:fcp-uniform} would not hold with
  the class of bounded-quantifier formulas
  of finite character in place of the class of $\Sigma^0_1$ such
  formulas, because in that case part (2) is provable in $\RCA$.  Thus,
  in spite of the similarity between the two classes suggested by the
  proof of Theorem~\ref{P:Sig1_RCA}, they do not coincide.
\end{rem}

\section{Finitary closure operators}\label{S:CE}

We can strengthen $\FCP$ by imposing additional requirements on the
maximal set being constructed. In particular, we now consider
requiring the maximal set to satisfy a finitary closure property as
well as a property of finite character.

\begin{defn}\label{D:CE} 
A \textit{finitary closure operator} is a set
of pairs $\langle F, n \rangle$ in which $F$ is (the canonical index
for) a finite (possibly empty) subset of $\N$ and $n \in \N$. A set $A
\subseteq \N$ is \textit{closed} under a finitary closure operator
$D$, or \emph{$D$-closed}, if for every $\langle F, n \rangle \in D$,
if $F \subseteq A$ then $n \in A$.
\end{defn}

\noindent This definition of a closure operator is not the standard set-theoretic
definition presented by Rubin and Rubin~\cite[Definition 6.3]{RR-1985}.
However, it is easy to see that for each operator of the one kind
there is an operator of the other such that the same sets are closed
under both.  Our definition has the advantage of being readily
formalizable in $\RCA$.

The following principle expresses the monotonicity of finitary closure
operators. The proof follows directly from definitions.

\begin{proposition}\label{p:clmonotone} It can be proved in $\RCAo$ that 
if $D$ is a finitary closure
operator and $A_0 \subseteq A_1 \subseteq A_2 \cdots$ is a sequence of
sets such that $\bigcup_{i \in \setN} A_i$ exists and 
each $A_i$ is $D$-closed, then $\bigcup_{i \in \N} A_i$
is $D$-closed.
\end{proposition}

The principle in the next definition is analogous to principle
$\mathsf{AL}' \, 3$ of Rubin and Rubin~\cite{RR-1985}, which is
equivalent to the axiom of choice in the context of set theory~\cite[p.~96, and Theorems~6.4 and~6.5]{RR-1985}.

\begin{defn} \label{def:ce}
The following scheme is defined in $\RCAo$.
\begin{list}{\labelitemi}{\leftmargin=0em}\itemsep2pt
\item[]($\AL$) If $D$ is a finitary closure operator, $\varphi$ is an $\Lang_2$
formula of finite character, and $A$ is any set, then every $D$-closed
subset of $A$ satisfying $\varphi$ is contained in a maximal such
subset.
\end{list}
\end{defn}

\noindent In the terminology of Rubin and Rubin~\cite{RR-1985}, this
is a ``primed'' statement, meaning that it asserts the existence not
merely of a maximal subset of a given set, but the existence of a
maximal extension of any given subset.  Primed versions of $\FCP$ and its restrictions can be formed, and are equivalent to the unprimed versions over $\RCAo$. By contrast, $\AL$ has only a
primed form.  This is because if $A$ is a set, $\varphi$ is a formula
of finite character, and $D$ is a finitary closure operator, $A$ need
not have any $D$-closed subset of which $\varphi$ holds.  For example,
suppose $\varphi$ holds only of $\emp$, and $D$ contains a pair of the
form $\langle \emp, a \rangle$ for some $a \in A$.

This leads to the observation that the requirements in the $\AL$
scheme that the maximal set must both be $D$-closed and satisfy a
property of finite character are, intuitively, in opposition to each
other.  Satisfying a finitary closure property is a positive
requirement, in the sense that forming the closure of a set usually
requires adding elements to the set. Satisfying a property of finite
character can be seen as a negative requirement in light of part~(1)
of Proposition~\ref{p:fcmonotone}.

We consider restrictions of $\AL$ as we did restrictions of $\FCP$
above.  By analogy, if $\Gamma$ is a class of formulas, we use the
notation $\Gamma\text{-}\AL$ to denote the restriction of $\AL$ to the
formulas in~$\Gamma$.  We begin with the following analogue of part~(1) of Theorem~\ref{thm_main_fcp} from the previous section.

\begin{thm}\label{T:mainchar_CE}
For $i \in \{0,1\}$ and $n \geq 1 $,
let $\Gamma$ be $\Pi^i_n$, $\Sigma^i_n$, or $\Delta^1_n$.
Then $\Gamma$-$\AL$ is provable in $\Gamma$-$\CA$.
\end{thm}

\begin{proof} 
Let $\varphi$ be a formula of finite character
in $\Gamma$, which may have parameters, and let $D$ be a finitary closure
 operator. Let $A$ be any set and
let $C$ be a $D$-closed subset of $A$ such that $\varphi(C)$ holds.
  
For any $X \subseteq A$, let $\cl_D(X)$ denote the \emph{$D$-closure}
of $X$.  That is, $\cl_D(X) = \bigcup_{i \in \N} X_i$, where
$X_0 = X$ and for each $i \in \N$, $X_{i+1}$ is the set of all $n \in
\N$ such that either $n \in X_i$ or there is a finite set $F \subseteq
X_i$ such that $\langle F,n \rangle \in D$.  Because we take
$D$ to be a set, $\cl_D(X)$ can be defined using a $\Sigma^0_1$ formula with 
parameter $D$.
Define a formula $\psi(k, X)$ by
\begin{align*}
\psi(k, X)  \Longleftrightarrow {} 
& (\forall n)[ ( D_n \subseteq \cl_D(X
\cup D_k ) \Longrightarrow \chat{\varphi}(n)] \\
&\wedge \cl_D(X \cup D_k) \subseteq A,
\end{align*}
where $\chat{\varphi}$ is as in Lemma~\ref{l:finiteset}. Note that 
$\psi$ is
arithmetical if $\Gamma$ is $\Pi^0_n$ or $\Sigma^0_n$, and is in
$\Gamma$ otherwise.

Define a function $f \colon \N \to \{0,1\}$ inductively such that $f(i)
= 1$ if and only if $\psi(\{j < i : f(j) = 1\} \cup \{i \}, C)$ holds. 
The characterization of the complexity of $\psi$ ensures that this $f$ can be constructed using $\Gamma$ comprehension, by first forming the oracle $\{ k : \psi(k, C)\}$.

Now, for each $i \in \N$, let
\[ 
B_i = \cl_D(C \cup \{ j < i : f(j) = 1\}),
\]
 and let $B = \bigcup_{i \in \N} B_i$.  The
construction of $f$ ensures that $\varphi(B_i)$ implies
$\varphi(B_{i+1})$ for all~$i \in \setN$, and we have assumed that $\varphi$ holds of
$B_0 = \cl_D(C) = C$. Therefore, an instance of induction
shows that $\varphi$ holds of $B_i$ for all $i \in \N$, and thus also
of $B$ by Proposition~\ref{p:fcmonotone}.  This also shows that $B
\subseteq A$.  Similarly, because each $B_i$ is $D$-closed, the
formalized version of Proposition~\ref{p:clmonotone} implies $B$ is
$D$-closed.

Finally, we check that $B$ is maximal.   Suppose that $H$ is a $D$-closed set
such that $B \subseteq H \subseteq A$ and $\varphi(H)$ holds. Fixing $i \in H$,
because 
$B_i \subseteq B \subseteq H$ and $H$ is $D$-closed, we have $\cl_D(B_i \cup \{i\}) \subseteq H$. 
 Thus, $\varphi(F)$ holds for every finite subset $F$ of
$\cl_D(B_i \cup \{i\})$, so by construction $f(i) = 1$ and $B_{i+1} =
\cl_D(B_i \cup \{i\})$. 
Because $B_{i+1} \subseteq B$, we conclude that $i
\in B$. Thus $B = H$, as desired.
\end{proof}

It follows that for most standard syntactical classes $\Gamma$,
$\Gamma\text{-}\AL$ is equivalent to $\Gamma\text{-}\FCP$.  Indeed,
for any class $\Gamma$ we have that $\Gamma\text{-}\AL$ implies
$\Gamma\text{-}\FCP$, because any instance of the latter can be regarded
as an instance of the former by adding an empty finitary closure
operator.  Conversely, if $\Gamma$ is $\Pi^0_n$, $\Pi^1_n$, $\Sigma^1_n$, or
$\Delta^1_n$, then $\Gamma\text{-}\FCP$ is equivalent to
$\Gamma\text{-}\CA$ by Theorem~\ref{thm_main_fcp}~(2), and hence
equivalent to $\Gamma\text{-}\AL$.  Thus, in particular, parts (2)--(5)
of Corollary~\ref{c:fcpstrength} hold for $\AL$ in place of $\FCP$,
and the full scheme $\AL$ itself is equivalent to $\mathsf{Z}_2$.

The proof of the preceding theorem does not work for $\Gamma =
\Delta^0_1$, because then $\Gamma\text{-}\CA$ is just $\RCA$, and we
need at least $\ACA$ to prove the existence of the function $f$
defined there (the formula $\psi(\sigma,X)$ being arithmetical at
best).  The next theorem shows that this cannot be avoided, even
for a class of considerably weaker formulas.

\begin{thm}\label{P:alqf_implies_aca} 
$\QF\text{-}\AL$ implies $\ACA$ over $\RCA$.
\end{thm}

\begin{proof}
Assume a one-to-one function $f \colon \N \to \N$ is given.
Let $\varphi(X)$ be the quantifier-free formula $0 \notin X$, which
trivially has finite character, and let \mbox{$\langle p_i: i \in \N
\rangle$} be an enumeration of all primes.  Let $D$ be the finitary
closure operator consisting, for all $i, n \in \N$, of all pairs of
the form
\begin{itemize}
\item $\langle \{p_i^{n+1}\},p_i^{n+2} \rangle$;

\item $\langle \{p_i^{n+2}\},p_i^{n+1} \rangle$;

\item $\langle \{p_i^{n+1}\},0 \rangle$, if $f(n) = i$.
\end{itemize} 
The set $D$ exists by $\Delta^0_1$ comprehension
relative to $f$ and our enumeration of primes.

Note that $\emp$ is a $D$-closed subset of $\N$ and $\varphi(\emp)$
holds.  Thus, we may apply $\AL$ for quantifier-free formulas to
obtain a maximal $D$-closed subset $B$ of $\N$ such that $\varphi(B)$
holds.  By definition of $D$, for every $i \in \N$, $B$ either
contains every positive power of $p_i$ or no positive power.  Now if
$f(n) = i$ for some~$n$, then no positive power of $p_i$ can be in $B$,
because otherwise $p_i^{n+1}$ would necessarily be in $B$ and hence so
would~$0$.  On the other hand, if $f(n) \neq i$ for all $n$ then $B
\cup \{p_i^{n+1} : n \in \N \}$ is $D$-closed and satisfies $\varphi$,
so by maximality $p^{n+1}_i$ must belong to $B$ for every~$n$.  It
follows that $i \in \operatorname{range}(f)$ if and only if $p_i \not\in
B$, so the range of $f$ exists.
\end{proof}

The next corollary can be contrasted with 
\ref{c:fcpstrength}~part~(1) and Theorem~\ref{P:Sig1_RCA} to illustrate a
difference between $\AL$ from $\FCP$ in terms of
some of their weakest restrictions. 

\begin{cor}\label{c:alequiv} 
The following are equivalent over $\RCA$:
\begin{enumerate}
\item $\ACA$;
\item $\Sigma^0_1\text{-}\AL$;
\item $\Sigma^0_0\text{-}\AL$;
\item $\QF\text{-}\AL$.
\end{enumerate}
\end{cor}

We conclude this section with one additional illustration of how
formulas of finite character can be used in conjunction with finitary
closure operators.  Recall the following concepts from order theory:
\begin{itemize}
\item a \textit{countable join-semilattice} is a countable poset
$\langle L, \leq_L \rangle $ with a maximal element $1_L$ and 
a join operation $\lor_L \colon L \times L \to L$ such that for all $a,b \in
L$, $a \lor_L b$ is the least
upper bound of $a$ and~$b$;

\item an \textit{ideal} on a countable join-semilattice $L$ is a
subset $I$ of $L$ that is downward closed under $\leq_L$ and closed
under $\lor_L$.

\end{itemize} 
The principle in the following proposition is the
countable analogue of a variant of $\mathsf{AL}' \, 1$ in Rubin and
Rubin~\cite{RR-1985}; compare with Proposition~\ref{P:NCE_ideals}
below.  For more on the computability theory of ideals on lattices,
see Turlington~\cite{Turlington-2010}.

\begin{prop}\label{p:alextend}
Over $\RCA$, $\QF\text{-}\AL$ implies that every proper
ideal on a countable join-semilattice extends to a maximal proper ideal.
\end{prop}

\begin{proof}
Let $L$ be a countable join-semilattice.  Let $\varphi$
be the formula $1 \not \in X$, and let $D$ be the finitary closure
operator consisting of all pairs of the form
\begin{itemize}
\item $\langle \{a,b\}, c\rangle$ where $a,b \in L$ and $c = a \lor b$;
\item $\langle \{a\}, b\rangle$, where $b \leq_L a$.
\end{itemize}
Because we define a join-semilattice to come with both
the order relation and the join operation, the set $D$ is $\Delta^0_0$
with parameters, so $\RCAo$ proves $D$ exists. It is immediate that a
set $X$ is closed under $D$ if and only if $X$ is an ideal in~$L$.
\end{proof}

We have not been able to prove a reversal corresponding to the previous proposition.

\begin{quest}
What is the strength of the principle asserting that every proper ideal on
a countable join-semilattice extends to a maximal proper ideal?
\end{quest}

\noindent  This question is further motivated by work of
Turlington~\cite[Theorem 2.4.11]{Turlington-2010}
on the similar problem of constructing prime
ideals on computable lattices.  However, because a maximal ideal on a
countable lattice need not be a prime ideal, Turlington's results do not
directly resolve our question.

\section{Nondeterministic finitary closure operators}\label{S:NCE}

It appears that the underlying reason that the restriction of $\AL$
to arithmetical formulas is provable in $\ACAo$ (and more generally,
why $\Gamma\text{-}\AL$ is provable in $\Gamma\text{-}\CA$ if $\Gamma$
is as in Theorem~\ref{T:mainchar_CE}) is that our definition of
finitary closure operator is very constraining.  Intuitively, if $D$
is such an operator and $\varphi$ is an arithmetical
formula, and we seek to extend some $D$-closed subset 
$B$ satisfying $\varphi$ to a maximal
such subset, we can focus largely on ensuring 
that $\varphi$ holds.  Achieving closure under $D$ is
relatively straightforward, because at each stage we only need to search through
all finite subsets $F$ of our current extension, and then adjoin all $n$
such that $\langle F,n \rangle \in D$.  This closure process becomes
far less trivial if we are given a choice of which elements
to adjoin. We now consider the case when each finite subset $F$
can be associated with a possibly infinite set of numbers  
from which we must choose at least one to adjoin. Intuitively, this
change adds an aspect of dependent choice when we wish to form the closure of a set. We will show that this weaker
notion of closure operator leads to a strictly stronger analogue of $\AL$.

\begin{defn}
A \textit{nondeterministic finitary closure operator} is a sequence 
of sets of the form $\langle F, S\rangle$ where $F$ is (the canonical index for)
a finite (possibly empty) subset 
of $\N$ and $S$ is a nonempty subset of~$\N$. A set $A \subseteq \N$ 
is \textit{closed} under a nondeterministic finitary closure operator 
$N$, or $N$-closed, if for each $\langle F, S \rangle$ in $N$, if 
$F \subseteq A$ then $A \cap S \neq \emp$.
\end{defn}

Note that if $D$ is a \emph{deterministic} finitary closure operator,
that is, a finitary closure operator in the stronger sense of the
previous section, then for any set $A$ there is a unique
$\subseteq$-minimal $D$-closed set extending~$A$. This is not true for
nondeterministic finitary closure operators. For example, let $N$ be the 
operator such that $\langle\emptyset,\N\rangle \in N$ and, for each $i
\in \N$ and each $j > i$, $\langle\{i\},\{j\}\rangle \in N$.  Then any
$N$-closed set extending $\emptyset$ will be of the form $\{i \in \N :
i \geq k\}$ for some~$k$, and any set of this form is $N$-closed. Thus
there is no $\subseteq$-minimal $N$-closed set.

In this section we study the following nondeterministic version of
$\AL$.

\begin{defn}\label{D:NCE}
The following scheme is defined in $\RCAo$.
\begin{list}{\labelitemi}{\leftmargin=0em}\itemsep2pt
\item[]($\mathsf{NCE}$) If $N$ is a nondeterministic closure operator,
$\varphi$ is an $\Lang_2$ formula of finite character, and $A$ is any set, then
every $N$-closed subset of $A$ satisfying $\varphi$ is contained in a
maximal such subset.
\end{list}
\end{defn}

\noindent 
Because the union of a
chain of $N$-closed sets is again $N$-closed, $\NCE$ can be
proved in set theory using Zorn's lemma.
 Restrictions of $\NCE$ to various syntactical classes of formulas are
defined as for $\AL$ and $\FCP$. 

\begin{rem}\label{rem:nce_remark} 
We might expect to be able to prove
$\NCE$ from $\AL$ by suitably transforming a given nondeterministic
finitary closure operator $N$ into a deterministic one.  For instance,
we could go through the members of $N$ one by one, and
for each such member $\langle F,S \rangle$ add $\langle F, n \rangle$
to $D$ for some $n \in S$ (e.g., the least $n$).  All $D$-closed sets
would then indeed be $N$-closed.  The converse, however, would not
necessarily be true, because a set could have $F$ as a subset for some
$\langle F,S \rangle \in N$, yet it could contain a different $n \in
S$ than the one chosen in defining~$D$.  In particular, a maximal
$D$-closed subset of a given set might not be maximal
among $N$-closed subsets. The results of this section 
demonstrate that it is impossible, in general, to 
reduce nondeterministic closure operators 
to deterministic ones in weak systems.
\end{rem}

Recall that an \emph{ideal} on a countable poset $\langle P,
\leq_P \rangle$ is a subset $I$ of $P$ downward closed under $\leq_P$
and such that for all $p,q \in I$ there is an $r \in I$ with $p \leq_P
r$ and $q \leq_P r$.  The next proposition is similar to
Proposition~\ref{p:alextend} above, which dealt with ideals on
countable join-semilattices.  In the proof of that proposition, we
defined a deterministic finitary closure operator $D$ in such a way
that $D$-closed sets were closed under the join operation.  For this
we relied on the fact that for every two elements in the semilattice
there is a unique element that is their join.  The reason we need
nondeterministic finitary closure operators below is that, for ideals
on countable posets, there are no longer unique elements witnessing the
relevant closure property.

\begin{prop}\label{P:NCE_ideals} 
Over $\RCAo$, $\QF\text{-}\mathsf{NCE}$ implies that every ideal on a countable
poset can be extended to a maximal ideal.
\end{prop}

\begin{proof} 
Let $\langle P, \leq_P \rangle$ be a
countable poset; without loss of generality we may assume $P$ 
is infinite.  Form an extended poset $\widehat{P}$ by adjoining a 
new element $t$ to $P$ and declaring $q <_{\widehat{P}} t$ for all $q \in P$. 
It follows immediately that the ideals on $P$ correspond exactly to the ideals of $\widehat{P}$ that do not contain $t$, 
and each ideal on $\widehat{P}$ which is maximal
among ideals not containing $t$ corresponds to a maximal ideal on~$P$.

Fix an enumeration $\{p_i : i \in \setN\}$ of $\widehat{P}$. 
 We form a nondeterministic closure
operator $N= \langle N_i : i \in \N\rangle$ such that, for each $i \in \N$,
\begin{itemize}
\item if $i = 2\langle j,k\rangle$ and $p_j \leq_{\widehat{P}} p_k$ then $N_i =
\langle \{p_k\},\{p_j\}\rangle$;

\item if $i = 2\langle j,k,l\rangle + 1$ and $p_j \leq_{\widehat{P}} p_l$ and $p_k
\leq_{\widehat{P}} p_l$ then 
\[ 
N_i = \langle \{p_j,p_k\}, \{p_n : (p_j \leq_{\widehat{P}} p_n) 
   \land (p_k \leq_{\widehat{P}} p_n)\}\rangle;
\]

\item otherwise,  $N_i = \langle\{p_i\},\{p_i\}\rangle$.

\end{itemize}
This construction gives a quantifier-free definition of
each $N_i$ uniformly in~$i$, so $\RCAo$ is able to construct~$N$.
Moreover, a subset of $\widehat{P}$ is $N$-closed if and only if it is an ideal.

Let $\varphi(X)$ be the formula $t \not\in X$, which is of finite character.
Fix an ideal $I \subseteq P$. Viewing $I$ as a subset of $\widehat{P}$, we see that $I$ is $N$-closed and $\varphi(I)$ holds. Thus,  by $\QF\text{-}\mathsf{NCE}$, there is a maximal $N$-closed extension $J \subseteq \widehat{P}$ satisfying~$\varphi$. This immediately yields a maximal ideal on $P$ extending $I$.
\end{proof}

Mummert~\cite[Theorem~2.4]{Mummert-2006} showed that the proposition
that every ideal on a countable poset extends to a maximal ideal is
equivalent to $\Pi^1_1\text{-}\CA$ over $\RCA$, which leads to the following corollary. 
This contrasts sharply with Theorem~\ref{T:mainchar_CE}, which showed
 that $\AL$ for arithmetical formulas is provable in $\ACA$. 
%

\begin{cor}\label{c:qfnce_strength}
$\QF\text{-}\NCE$ implies $\Pi^1_1\text{-}\CA$ over $\RCAo$.  
\end{cor}

We will state the precise strength of $\QF\text{-}\NCE$ in Corollary~\ref{t:ncereverse}
below.
We must first prove the following upper bound.
 The proof uses a
technique involving countable coded $\beta$-models, parallel to
Lemma~2.4 of Mummert~\cite{Mummert-2006}.  In $\ACA$, a
\emph{countable coded $\beta$-model} is defined as a sequence $\M =
\langle M_i : i \in \N \rangle$ of subsets of $\N$ such that for every $\Sigma^1_1$
formula $\varphi$ with parameters from $\M$, $\varphi$ holds if and
only if $\M \models \varphi$. $\Pi^1_1\text{-}\CA$ proves that every set
is included in some countable coded $\beta$-model. Complete information on 
countable coded $\beta$-models is given by Simpson~\cite[Section~VII.2]{Simpson-2009}.

\begin{thm}\label{p:nceprovable} 
$\Sigma^1_1\text{-}\mathsf{NCE}$ is provable in $\Pi^1_1\text{-}\CA$.
\end{thm}

\begin{proof}
Let $\varphi$ be a
$\Sigma^1_1$ formula of finite character (possibly with parameters)
and let $N$ be a nondeterministic closure operator. Let $A$ be any set
and let $C$ be an $N$-closed subset of $A$ such that $\varphi(C)$
holds.

Let $\M = \langle M_i : i \in \N\rangle$ be a countable coded
$\beta$-model containing $A$, $C$, $N$, and any parameters of
$\varphi$.  Using $\Pi^1_1$ comprehension,
we may form the set $\{i : \M \models \varphi(M_i)\}$.

Working outside $\M$, we build an increasing sequence $\langle B_i :
i \in \N\rangle$ of $N$-closed extensions of~$C$. Let $B_0 = C$.
Given~$i$, ask whether there is a $j$ such that
\begin{itemize}
\item $M_j$ is an $N$-closed subset of $A$;
\item $B_i \subseteq M_j$;
\item $i \in M_j$;
\item and $\varphi(M_j)$ holds.
\end{itemize} 
If there is, choose the least such $j$ and let
$B_{i+1} = M_j$. Otherwise, let $B_{i+1} = B_i$.  Finally, let $B =
\bigcup_{i\in \N} B_i$.

Because the inductive construction only asks arithmetical questions
about $\M$, it can be carried out in $\Pi^1_1\text{-}\CA$, and so
$\Pi^1_1\text{-}\CA$ proves that $B$ exists. Clearly $C \subseteq B
\subseteq A$.  An arithmetical induction shows that for all $i \in
\N$, $\varphi(B_i)$ holds and $B_i$ is $N$-closed.  Therefore, the
formalized version of Proposition~\ref{p:fcmonotone} shows that
$\varphi(B)$ holds, and the analogue of Proposition~\ref{p:clmonotone}
for nondeterministic finitary closure operators shows that $B$ is
$N$-closed.

Now suppose that $H$ is an $N$-closed
set such that $B \subseteq H \subseteq A$ and~$\varphi(H)$ holds. 
Fix $i \in H$.  Because $\varphi$
is $\Sigma^1_1$, the property
\begin{equation}\label{eq:betamod} 
(\exists X)[X \text{ is
$N$-closed} \land B_i \subseteq X \subseteq A \land i \in X \land
\varphi(X)]
\end{equation}
\noindent is expressible by a $\Sigma^1_1$ sentence with parameters from $\M$,
 and $H$ witnesses that it is true.
  Thus, because $\M$ is a
$\beta$-model, this sentence must be satisfied by $\M$, which means that
some $M_j$ must also witness it. The inductive construction must
therefore have selected such an $M_j$ to be $B_{i+1}$, which means $i
\in B_{i+1}$ and hence $i \in B$.  It follows that $B$ is maximal.
\end{proof}


We can now characterize the strength of $\Sigma^1_1\text{-}\NCE$ and its restrictions.

\begin{cor}\label{t:ncereverse} 
For each $n \geq 1$, the following are equivalent over $\RCAo$:
\begin{enumerate}
\item $\Pi^1_1\text{-}\CA$;
\item $\Sigma^1_1\text{-}\mathsf{NCE}$;
\item $\Sigma^0_n\text{-}\mathsf{NCE}$;
\item $\QF\text{-}\mathsf{NCE}$.
\end{enumerate}
\end{cor}

\begin{proof} 
Theorem~\ref{p:nceprovable} shows that (1) implies (2), and it is
obvious that (2) implies (3) and (3) implies (4).
 Corollary~\ref{c:qfnce_strength} shows that (4) implies~(1).
\end{proof}

Our final results characterize the strength of
$\mathsf{NCE}$ for formulas higher in the analytical
hierarchy. 


\begin{thm}\label{P:ncehigher}
For each $n \geq 1$,
\begin{enumerate}
\item $\Sigma^1_n\text{-}\mathsf{NCE}$ and $\Pi^1_n\text{-}\mathsf{NCE}$
are provable in  $\Pi^1_n\text{-}\mathsf{CA}_0$;
\item $\Delta^1_n\text{-}\mathsf{NCE}$ is provable in
$\Delta^1_n\text{-}\mathsf{CA}_0$.
\end{enumerate}
\end{thm}
\begin{proof}

We prove part (1), the proof of part (2) being similar.  Let $\varphi(X)$
be a $\Sigma^1_n$ formula of finite character, 
respectively a $\Pi^1_n$ such formula.  Let
$N$ be a nondeterministic closure operator,  let $A$
be any set, and let $C$ be an $N$-closed subset of $A$ such
that $\varphi(C)$ holds.

By Lemma~4.5, let $\widehat{\varphi}$ be a $\Sigma^1_n$
formula, respectively a $\Pi^1_n$ formula, such that
\[
(\forall X)(\forall n)[X = D_n \Longrightarrow (\varphi(X)
\Longleftrightarrow \widehat{\varphi}(n))].
\]
We may use $\Pi^1_n$ comprehension to
form the set $W = \{ n : \widehat{\varphi}(n)\}$. Define 
$\psi(X)$ to be the arithmetical formula $(\forall n)[D_n \subseteq X \Longrightarrow n \in
W]$. 

We claim that for  every set $X$, $\psi(X)$ holds if and
  only if $\varphi(X)$ holds.  The definitions of $W$ and
$\psi$
ensure that $\psi(X)$ holds if and only if $\varphi(D_n)$ holds for
every finite $D_n \subseteq X$, which is true if and only if
$\varphi(X)$ holds because $\varphi$ has finite character.  This establishes
the claim.

By the claim, $\psi$ is a property of finite character and
$\psi(C)$ holds. Using $\Sigma^1_1\text{-}\mathsf{NCE}$, which is provable in
$\Pi^1_1\text{-}\mathsf{CA}_0$ by Theorem~\ref{p:nceprovable} and thus is provable in
$\Pi^1_n\text{-}\mathsf{CA}_0$, there is a maximal
$N$-closed subset $B$ of $A$ extending~$C$ with
property~$\psi$. Again by the claim, $B$ is a maximal
$N$-closed subset of $A$ extending $B$ with
property~$\varphi$.
\end{proof}
\newpage

\begin{cor}
The following are provable in $\RCAo$:
\begin{enumerate}
\item for each $n \geq 1$, $\Delta^1_n\text{-}\CA$ is equivalent to
$\Delta^1_n\text{-}\mathsf{NCE}$;
\item for each $n \geq 1$, $\Pi^1_n\text{-}\CA$ is equivalent to
$\Pi^1_n\text{-}\mathsf{NCE}$ and to $\Sigma^1_n\text{-}\mathsf{NCE}$;
\item $\mathsf{Z}_2$ is equivelent to $\NCE$.
\end{enumerate}
\end{cor}
\begin{proof}
The implications from $\Delta^1_n\text{-}\CA$, $\Pi^1_n\text{-}\CA$, and
$\mathsf{Z}_2$ follow by Theorem~\ref{P:ncehigher}. On the other hand, 
each restriction of $\NCE$ trivially implies the corresponding restriction of $\FCP$,
so the reversals follow by Corollary~\ref{c:fcpstrength}.
\end{proof}

\begin{rem}
  The characterizations in this section shed light on the role of the
  closure operator in the principles $\AL$ and $\NCE$. For
  $n \geq 1$, we have shown that $\Sigma^1_n\text{-}\FCP$,
  $\Sigma^1_n\text{-}\AL$, and $\Sigma^1_n\text{-}\NCE$ are
  all equivalent over $\RCAo$. However, $\QF\text{-}\FCP$ is
  provable in $\RCAo$, $\QF\text{-}\AL$ is equvalent to
  $\ACAo$ over $\RCAo$, and $\QF\text{-}\NCE$ is equivalent
  to $\Pi^1_1\text{-}\CA$ over $\RCAo$. Thus the closure
  operators in the stronger principles serve as a sort of
replacement for arithmetical quantification in the case of
$\AL$, and for $\Sigma^1_1$ quantification in the case of
$\NCE$. This allows these principles to have greater strength
than might be suggested by the property of finite character
alone.   At higher levels of the analytical hierarchy, the
principles become equivalent because the 
complexity of the property of finite character overtakes the
complexity of the closure notions.
\end{rem}

\bibliographystyle{amsplain} \bibliography{Choice}

\end{document}